\documentclass[a4paper,12pt,leqno]{article}

\usepackage{amssymb,amsmath} % do not use if using class amsproc
\usepackage{amsthm}
\usepackage[abbrev]{amsrefs}
\usepackage{xy}
\xyoption{all}
\usepackage{graphicx}
\usepackage{hyperref}
\usepackage[text={15.5cm,25cm}]{geometry}

\newtheorem{defn}{Definition}[section]
\newtheorem{exm}[defn]{Example}
\newtheorem{theorem}[defn]{Theorem}
\newtheorem*{theorem*}{Theorem}

\newtheorem{lemma}[defn]{Lemma}

\newtheorem{corollary}[defn]{Corollary}
\newtheorem{rem}[defn]{Remark}

\numberwithin{equation}{section}

\newenvironment{remark}{\begin{rem}\em}{\end{rem}}

\newcommand\V{\bigvee}
\newcommand\ie{i.e.}
\newcommand\st{\mid}
\newcommand\topology{\operatorname{\Omega}}
\newcommand\spectrum{\operatorname{\Sigma}}
\newcommand\wi{\eqslantless}
\newcommand\Top{\textit{Top}}
\newcommand\Loc{\textit{Loc}}
\newcommand\ext{\operatorname{ext}}

\begin{document}

\title{Scott locales\thanks{Work funded by FCT/Portugal through project UID/4459/2025.}}
\author{Pedro Resende and Jo\~{a}o Paulo Santos}

\date{~}

\maketitle

\vspace*{-1cm}
\begin{abstract}
We prove some facts about locales $L$ equipped with the Scott topology $\topology(L)$, in particular studying a canonical frame homomorphism $\phi:\topology(L)\to L$ which is motivated by an application to cognitive science. Such a topological locale $L$ is called a \emph{Scott locale} if the inclusion of primes $p:\spectrum(L)\to L$ is continuous. We prove that the spectrum $\spectrum(L)$ of a Scott locale $L$ is necessarily $T_1$, and that preregular locales (a generalization of regular locales) are Scott locales. If $L$ is the topology of a topological space $X$ we find a (necessarily unique) continuous map $f:X\to L$ such that $f^{-1}=\phi$ and compare it with the points-to-primes map $p:X\to L$, showing that $f=p$ if and only if $X$ is preregular, and that a sober space $X$ is Hausdorff if and only if $X$ is $T_1$ and $f(X)\subseteq\spectrum(L)$.
\vspace{0.2cm}

\end{abstract}

\section{Introduction}\label{sec:intro}

Let $L$ be a locale equipped with a topology $\topology(L)$. The latter is another locale, and the question of how it relates to $L$ may be relevant in some contexts. For instance, examples of topological quantales, and, in particular, topological locales, occur as spaces of physical measurements in~\cites{fop,msnew}, and in the model of consciousness of~\cite{qualia} a related interpretation is offered of some topological quantales as being spaces of \emph{qualia} (intrinsic qualities of moments of subjective experience---cf.\ \cite{MindWorldOrder}) whose open sets are the \emph{pure concepts}, meaning---borrowing ideas from computer science~\cite{topologyvialogic}---the portions of finite information that can be recorded and communicated between conscious beings by finite means. Then, in this context, a homomorphism of locales
\[
\phi:\topology(L)\to L
\]
(by which is meant a \emph{frame} homomorphism)
may describe a mechanism through which a conscious being evokes concepts, an instance of what psychologists refer to as semantic memory~\cite{Kumar2021}. For instance, if $\mathcal U\in\topology(L)$ is the concept ``color red'' stored in Alice's synapses, $\phi(\mathcal U)$ would be the subjective experience that occurs when Alice evokes that concept.

A specific example discussed in~\cite{qualia} is the following, which defines a homomorphism whenever $\topology(L)$ is contained in the Alexandrov topology of $L$, as we will see below:
\begin{equation}\label{phi}
\phi(\mathcal U)=\V\{a\in L\st\exists_{b\in \mathcal U}\ a\wedge b=0\}.
\end{equation}

If, as in~\cite{qualia}, $L$ is a locally compact locale then, with the Scott topology, it is a sober space. Hence, identifying $L$ with the topology $\topology(X)$ of a locally compact sober space $X$, the homomorphisms $\phi:\topology(L)\to L$ are in bijection with the continuous maps
\[
f:X\to L.
\]
These represent the points of $X$ continuously by open sets of $X$. Another representation of points by open sets, but not necessarily continuous, is the mapping of points to primes
\[
p:X\to L,
\]
which for all $x\in X$ is given by $p(x)=X\setminus\overline{\{x\}}$.

The main purpose of this paper is to compare $f$ and $p$ when $\phi$ is given by \eqref{phi}, in particular because the situations where $f=p$ are especially relevant in~\cite{qualia}. But, while local compactness is a given in the latter, in this paper it is not needed, and we prove more general results concerning preregular spaces, not requiring sobriety of $L$ when equipped with the Scott topology, but still allowing a unique continuous map $f$ to be defined such that $f^{-1}=\phi$, and leading to the following conclusion:

\begin{theorem*}
A $T_0$ space $X$ is Hausdorff if and only if $f=p$.
\end{theorem*}

In addition, we introduce the notion of \emph{Scott locale}, by which is meant a locale $L$ whose prime spectrum inclusion $\spectrum(L)\to L$ is continuous with respect to the Scott topology on $L$. We will see that this makes $\spectrum(L)$ a $T_1$ space, and that examples of Scott locales are given by preregular locales---a weakening of regular locales.

These results provide a small addition to the extensive body of work on separation axioms for locales, in particular to the problem of finding localic analogues of Hausdorff spaces~\cite{picadopultr-separation}.

For general definitions and facts concerning sober spaces and locales we refer to~\cites{picadopultr,stonespaces}.

\section{The homomorphism $\phi$}

Let $L$ be an arbitrary locale equipped with a topology $\topology(L)$, and let $\phi:\topology(L)\to L$ be the mapping defined in \eqref{phi}.
Define the following notation, for each $\mathcal U\in\topology(L)$:
\[
\mathcal D_{\mathcal U} := \bigl\{ a\in L\mid \exists_{b\in \mathcal U}\ a\wedge b=0\bigr\}.
\]
Then we have
\[
\phi(\mathcal U) = \V\mathcal D_{\mathcal U}.
\]

As stated in the introduction, the conclusion that $\phi$ is a homomorphism of locales holds under very general assumptions:

\begin{lemma}
If $\topology(L)$ is contained in the Alexandrov topology of $L$ then $\phi:\topology(L)\to L$ is a homomorphism of locales.
\end{lemma}

\begin{proof}
We need to prove the following three conditions:
\begin{enumerate}
\item\label{case1} $\phi(L)=1$.
\item\label{case2} $\phi(\bigcup_i \mathcal U_i)=\V_i\phi(\mathcal U_i)$ for all families $(\mathcal U_i)$ in $\topology(L)$.
\item\label{case3} $\phi(\mathcal U\cap \mathcal V)=\phi(\mathcal U)\wedge\phi(\mathcal V)$ for all $\mathcal U,\mathcal V\in\topology(L)$.
\end{enumerate}

\begin{trivlist}
\item Proof of \eqref{case1}. Immediate because $1\in L$ and $1\wedge 0=0$.

\item Proof of \eqref{case2}. For any family of open sets $(\mathcal U_i)$ we have
\[
\phi\bigl(\bigcup_i \mathcal U_i\bigr) = \V\mathcal D_{\bigcup_i \mathcal U_i}
=\V\bigcup_i\{a\in L\st\exists_{x\in \mathcal U_i}\ a\wedge x=0\}=\V_i\phi(\mathcal U_i).
\]

\item Proof of \eqref{case3}. Let $\mathcal U,\mathcal V\in\topology(L)$. Due to \eqref{case2} $\phi$ is monotone, so we have
\[
\phi(\mathcal U\cap \mathcal V)\le \phi(\mathcal U)\wedge\phi(\mathcal V).
\]
Hence, in order to obtain \eqref{case3} only the converse inequality needs to be proved:
\[
\phi(\mathcal U)\wedge\phi(\mathcal V)\le\phi(\mathcal U\cap \mathcal V).
\]
Using locale distributivity we obtain
\[
\phi(\mathcal U)\wedge\phi(\mathcal V)=\V\{a\wedge b\st\exists_{x\in \mathcal U}\exists_{y\in \mathcal V}\ a\wedge x=b\wedge y=0\}.
\]
For each $a,b\in L$ and each $x\in \mathcal U$ and $y\in \mathcal V$ such that $a\wedge x=b\wedge y=0$ we have
\[
a\wedge b\wedge(x\vee y) = (a\wedge b\wedge x)\vee (a\wedge b\wedge y)=0,
\]
and $x\vee y\in \mathcal U\cap \mathcal V$ because both $\mathcal U$ and $\mathcal V$ are upper-closed, so
$a\wedge b\le\phi(\mathcal U\cap \mathcal V)$. Therefore $\phi(\mathcal U)\wedge\phi(\mathcal V)$, which is the supremum of all such elements $a\wedge b$, is less or equal to $\phi(\mathcal U\cap \mathcal V)$. \qedhere
\end{trivlist}
\end{proof}

Under the same assumptions, $\mathcal D_{\mathcal U}$ and $\phi(\mathcal U)$ can equivalently be defined using the pseudo-complement operation on $L$, which for each $a\in L$ will be denoted by
\[
\neg a = \V\{b\in L\mid a\wedge b=0\}:
\]

\begin{lemma}\label{remark}
If $\topology(L)$ is contained in the Alexandrov topology then, for all $\mathcal U\in\topology(L)$,
\[
\mathcal D_{\mathcal U} = \bigl\{ a\in L\mid \neg a\in\mathcal U\bigr\}.
\]
Moreover, $\mathcal D_{\mathcal U}$ is downwards closed in $L$, and $\mathcal D_{\mathcal U}$ is an ideal if $\mathcal U$ is a filter.
\end{lemma}

\begin{proof}
If $\neg a\in\mathcal U$ then, making $b=\neg a$, we obtain $a\wedge b=0$ and $b\in\mathcal U$ (this does not require knowing that $\mathcal U$ is upwards closed). Conversely, if there is $b\in L$ such that $a\wedge b=0$ and $b\in\mathcal U$, then, since $\mathcal U$ is upwards closed and the condition $a\wedge b=0$ is equivalent to $b\le\neg a$, we obtain $\neg a\in\mathcal U$.

For the second part, let $a\le b\in\mathcal D_{\mathcal U}$. Then $\neg b\in\mathcal U$ and $\neg b\le\neg a$, so $\neg a\in\mathcal U$ because $\mathcal U$ is upwards closed; that is, $a\in\mathcal D_{\mathcal U}$. Now assume that $\mathcal U$ is a filter. Then $\mathcal U\neq\emptyset$, so $1\in\mathcal U$ and therefore $0\in \mathcal D_{\mathcal U}$, showing that $\mathcal D_{\mathcal U}\neq\emptyset$. Finally, if $a,b\in \mathcal D_{\mathcal U}$ then $\neg a\in\mathcal U$ and $\neg b\in\mathcal U$, so
\[
\neg(a\vee b)=\neg a\wedge\neg b\in\mathcal U,
\]
and we conclude that $a\vee b\in\mathcal D_{\mathcal U}$.
\end{proof}

We list two more alternative definitions of $\phi$, although they will not be needed in the paper:

\begin{remark}
For each $\mathcal U\in \topology(L)$ define the set
\[
\mathcal B_{\mathcal U} := \bigl\{ \neg b\mid b\in\mathcal U\bigr\} = \bigl\{a\in L\mid \exists_{b\in\mathcal U}\ a=\neg b\bigr\}.
\]
Since the condition $a=\neg b$ is stronger than $a\le \neg b$, which is equivalent to $a\wedge b=0$, we obtain $\mathcal B_{\mathcal U}\subseteq\mathcal D_{\mathcal U}$, and thus
\[
\V\mathcal B_{\mathcal U}\le\phi(\mathcal U).
\]
Conversely, if $a\in\mathcal D_{\mathcal U}$ we have $a\wedge b=0$ for some $b\in\mathcal U$. Hence, for each $a\in\mathcal D_{\mathcal U}$ there is $b\in L$ such that $a\le\neg b\in\mathcal B_{\mathcal U}$, and thus
\[
\phi(\mathcal U)\le\V\mathcal B_{\mathcal U}.
\]
Therefore $\phi$ can be defined as in \S\ref{remark} but swapping $a$ and $\neg a$:
\[
\phi(\mathcal U) = \V\mathcal B_{\mathcal U}=\{\neg b\mid b\in\mathcal U\}.
\]
Notice that this holds regardless of any conditions on $\topology(L)$.

Still another equivalent definition is derived directly from \S\ref{remark}, where $L_{\neg\neg}$ is the Booleanization of $L$---\ie, the frame quotient of $L$ by the double negation nucleus $\neg\neg$:
\[
\phi(\mathcal U)=\V\{a\in L_{\neg\neg}\st \neg a\in\mathcal U\}.
\]
Clearly, the right hand side is less or equal to $\phi(\mathcal U)$ because $\{a\in L_{\neg\neg}\st\neg a\in\mathcal U\}\subseteq\mathcal D_{\mathcal U}$. For the converse inequality, for each $a\in\mathcal D_{\mathcal U}$ we have $a\le\neg\neg a\in L_{\neg\neg}$ and $\neg\neg\neg a=\neg a\in\mathcal U$, so $\phi(\mathcal U)\le\V\{a\in L_{\neg\neg}\st \neg a\in\mathcal U\}$.
\end{remark}

\section{preregular spaces}

Now let $X$ be an arbitrary topological space, and let $L:=\topology(X)$ be its locale of open sets, whose elements now will be denoted by $U,V$. Let also $\topology(L)$ denote the Scott topology on $L$. For each $U\in L$ the pseudo-complement $\neg U$ is the exterior of $U$:
\[
\neg U=\ext U:= X\setminus\overline U.
\]
Hence, by \S\ref{remark}, the map $\phi:\topology(L)\to L$ of \eqref{phi} is given by
\begin{eqnarray*}
\phi(\mathcal U) &=& \bigcup\mathcal D_{\mathcal U}\\
&=& \bigcup\bigl\{ U\in L\mid \exists_{V\in\mathcal U}\ U\cap V=\emptyset\bigr\}\\
&=& \bigcup\bigl\{U\in L\mid \ext U\in\mathcal U\bigr\}.
\end{eqnarray*}

For each $x\in X$ write $\mathcal N_x$ for the set of all the open neighborhoods of $x$. This is an open set of the Scott topology of $L$, so we can define a map $f:X\to L$ by
\[
f(x) = \phi(\mathcal N_x).
\]
The following equivalences are immediate:
\begin{equation}\label{eqfx}
y\in f(x)\iff\bigl(\exists_{U\in\mathcal N_x}\exists_{V\in\mathcal N_y}\ U\cap V=\emptyset\bigr)\iff x\in f(y).
\end{equation}
The set $\mathcal N_x$ is also a filter, so $\mathcal D_{\mathcal N_x}$ is a directed set, by \eqref{remark}.

\begin{lemma}
For all $\mathcal U\in\topology(L)$ we have $f^{-1}(\mathcal U)=\phi(\mathcal U)$, so $f$ is a continuous map.
\end{lemma}

\begin{proof}
The proof follows from the following straightforward derivation, which uses the fact that $\mathcal D_{\mathcal N_x}$ is directed and $\mathcal U$ is Scott open:
\begin{eqnarray*}
x\in f^{-1}(\mathcal U) &\iff& f(x)\in\mathcal U
\iff \phi(\mathcal N_x)\in\mathcal U\\
&\iff& \bigcup\mathcal D_{\mathcal N_x}\in\mathcal U \iff \exists_{U\in\mathcal D_{\mathcal N_x}}\ U\in\mathcal U\\
%&\iff& \exists_{U\in L}\exists_{V\in\mathcal N_x}\ U\cap V=\emptyset\text{ and }U\in\mathcal U\\
&\iff& \exists_{U\in\mathcal U}\exists_{V\in\mathcal N_x}\ U\cap V=\emptyset\\
&\iff& x\in\bigcup\{V\in L\st \exists_{U\in\mathcal U}\ U\cap V=\emptyset\}\\
&\iff& x\in\bigcup\mathcal D_{\mathcal U}
\iff x\in\phi(\mathcal U). \qedhere
\end{eqnarray*}
\end{proof}

Now let us compare $f$ to the map $p:X\to L$ that sends each $x\in X$ to the corresponding prime element
\[
p(x) = X\setminus\overline{\{x\}}.
\]
\begin{lemma}\label{lem:flep}
For all $x\in X$ we have $f(x)\subseteq p(x)$.
\end{lemma}

\begin{proof}
Let $y\in f(x)$. Then by \eqref{eqfx} there is $U\in\mathcal N_y$ such that $x\notin U$; that is, $y\notin\overline{\{x\}}$. Therefore, $y\in p(x)$.
\end{proof}

Two points $x,y\in X$ are topologically indistinguishable if they have the same open neighborhoods: $\mathcal N_x=\mathcal N_y$. This is equivalent to the condition $\overline{\{x\}}=\overline{\{y\}}$ because $z\in\overline{\{x\}}$ if and only if $\mathcal N_z\subseteq\mathcal N_x$.
Recall that a topological space is said to be \emph{preregular}, or $R_1$, if every pair of topologically distinguishable points can be separated by disjoint neighborhoods.
This implies that the specialization order of any preregular space is an equivalence relation, and that any $T_0$ preregular space is Hausdorff.

\begin{theorem}\label{lem:prereg}
A topological space $X$ is preregular if and only if $f=p$.
\end{theorem}

\begin{proof}
Assume that $X$ is preregular. By the previous lemma we only need to prove that $p(x)\subseteq f(x)$ for all $x\in X$. 
Let $y\in p(x)$. Then $y\notin\overline{\{x\}}$, which means that $x$ and $y$ are topologically distinguishable, so there are $U\in\mathcal N_x$ and $V\in\mathcal N_y$ such that $U\cap V\neq \emptyset$. By \eqref{eqfx}, we conclude $y\in f(x)$.

For the converse implication, assume that $p=f$ and let $x,y\in X$ be topologically distinguishable. Without loss of generality, assume $y\notin\overline{\{x\}}$. Then $y\in p(x)=f(x)$, so by \eqref{eqfx} the points $x$ and $y$ can be separated by disjoint neighborhoods.
\end{proof}

\begin{corollary}\label{cor:T0feqp}
A $T_0$ topological space $X$ is Hausdorff if and only if $f=p$.
\end{corollary}

\section{Scott locales}

In this section we assume that $L$ is an arbitrary locale, equipped with the Scott topology $\topology(L)$. Let us denote by $X:=\spectrum(L)$ the spectrum of prime elements of $L$. The open sets of the spectrum are of the form $\mathcal U_a$ where for each $a\in L$
\[
\mathcal U_a=\{x\in X\st a\nleq x\}.
\]
Now the points-to-primes map is the inclusion map $p:X\to L$. We say that $L$ is a \emph{Scott locale} if $p$ is continuous.

\begin{lemma}\label{lem:T1}
If $L$ is a Scott locale then $X$ is a $T_1$ space.
\end{lemma}

\begin{proof}
Let $x,y\in X$ be such that $x\sqsubseteq y$ in the specialization order of the topology of $X$. The continuity of the embedding implies that we have $x\sqsubseteq y$ also with respect to the topology of $L$; that is, we have $x\le y$ in the locale order. However, the condition $x\sqsubseteq y$ in $X$ is equivalent to $y\le x$, and therefore $x=y$.
\end{proof}

In particular, from this lemma it follows that the map $p$ in section~\ref{sec:intro} is not necessarily continuous, since locally compact sober spaces are not necessarily $T_1$.

Let us find a sufficient condition for a locale to be a Scott locale.
A weakening of the definition of regularity for locales is obtained as follows: we say that a locale $L$ is \emph{preregular} if the condition
\begin{equation}\label{R1}
x = \V_{a \wi x} a
\end{equation}
holds for all $x\in X$, where $\wi$ is the \emph{well inside} relation~\cite{stonespaces}*{III-1.1}. It is easy to see that if this condition holds then $X$ is a preregular space. Since the spectrum of a locale is always $T_0$, this implies that $X$ is Hausdorff. Conversely, if $Y$ is a preregular space (in particular, if it is Hausdorff) the locale $\topology(Y)$ is preregular. Hence, the adjunction between $\Top$ and $\Loc$ restricts to an adjunction between the full subcategories of preregular topological spaces and preregular locales, and it further restricts to an adjunction between the category of Hausdorff spaces and that of preregular locales.

\begin{theorem}\label{1stthm}
If $L$ is preregular then it is a Scott locale.
\end{theorem}

\begin{proof}
Let $L$ be preregular.
Let $x\in \mathcal U$ for some $x\in X$ and $\mathcal U\in\topology(L)$. Let $\mathcal D\subseteq L$ be the directed set of elements well inside $x$. Then $\V \mathcal D=x$ and there is $a\in \mathcal D$ such that $a\in \mathcal U$. Then $\neg a\vee x=1$, and thus $\neg a\nleq x$ because $x\neq 1$ (since $x$ is prime); that is, $x\in \mathcal U_{\neg a}$. Let $y\in \mathcal U_{\neg a}$, \ie, $\neg a\nleq y$. Then, since $y$ is prime and $\neg a\wedge a=0\le y$, we must have $a\le y$, and thus $y\in \mathcal U$. Hence, $x\in \mathcal U_{\neg a}\subseteq \mathcal U$, showing that $p:X\to L$ is continuous.
\end{proof}

We can bring together the results of this section and the previous one by focusing on a sober space $X$ whose topology $L=\topology(X)$ is equipped with the Scott topology $\topology(L)$, where $f:X\to L$ is defined as before and we write $p:X\to L$ for the points-to-primes map.

\begin{theorem}
Let $X$ be a sober space whose topology $L=\topology(X)$ is equipped with the Scott topology $\topology(L)$. The following conditions are equivalent.
\begin{enumerate}
\item\label{main0} $X$ is Hausdorff.
\item\label{main3} $X$ is $T_1$ and $f(x)$ is prime for all $x\in X$.
\end{enumerate}
\end{theorem}

\begin{proof}
$\eqref{main0}\Rightarrow \eqref{main3}$. If $X$ is Hausdorff then $p=f$, by \S\ref{cor:T0feqp}, so $p$ is continuous. Hence, $L$ is a Scott locale and, by \S\ref{lem:T1}, $X$ is $T_1$.

$\eqref{main3}\Rightarrow \eqref{main0}$. Assume that $X$ is $T_1$ and $f(X)\subseteq\spectrum(L)$.
Then for all $x\in X$ we have $f(x)=p(z)$ for some $z\in X$, and $p(z)\subseteq p(x)$ due to \S\ref{lem:flep}. Hence, $x\in\overline{\{z\}}$, so $x=z$ because $X$ is $T_1$. Therefore $f(x)=p(x)$ for all $x\in X$ and, by \S\ref{cor:T0feqp}, $X$ is Hausdorff.
\end{proof}

\vspace*{5mm}
\noindent {\sc
Centro de An\'alise Matem\'atica, Geometria e Sistemas Din\^amicos\\
Departamento de Matem\'{a}tica, Instituto Superior T\'{e}cnico\\
Universidade de Lisboa,
Av.\ Rovisco Pais 1, 1049-001 Lisboa, Portugal}\\
{\it E-mail:} {\sf pedro.m.a.resende@tecnico.ulisboa.pt,\ jsantos@math.tecnico.ulisboa.pt}

\end{document}